%% file: main.tex
\documentclass[12pt,a4paper]{article}

\usepackage[biblatex]{anziamjedraft}
\bibliography{manu}

\usepackage{subcaption}
\usepackage{adjustbox}
\usepackage{amsfonts}       

\usepackage{pgfplots}
\pgfplotsset{compat=1.15}
\usepackage{pgfkeys}
\usetikzlibrary{fillbetween}
\usetikzlibrary{backgrounds}

\newcommand{\R}{\mathbb{R}}
\newcommand{\N}{\mathbb{N}}

\renewcommand{\det}{\operatorname{det}}

\newcommand{\trace}{\operatorname{trace}}

\newcommand{\lmax}{\text{max}}
\newcommand{\E}{\operatorname{\mathcal{E}}}

\title{Low rank approximation of positive semi-definite symmetric matrices using Gaussian elimination and volume sampling}

\author{M. Hegland}
\address{Mathematical Sciences Institute, Australian National University,
ACT~2601, \textsc{Australia}.}
\mailto{markus.hegland@anu.edu.au}
\myorcid{0000-0002-5136-2883}

\author{F. de~Hoog}
\address{Data61, CSIRO, GPO Box 1700, Canberra,  ACT~2601,\textsc{Australia}}
\mailto{frank.dehoog@csiro.au}
\myorcid{0000-0002-4632-564X}
\date{1 December 2020}

\begin{document}
\maketitle

\begin{abstract}
Positive semi-definite matrices commonly occur as normal matrices of least squares problems in statistics or as kernel matrices in machine learning and approximation theory. They are typically large and dense. Thus algorithms to solve systems with such a matrix can be very costly. A core idea to reduce computational complexity is to approximate the matrix by one with a low rank. The optimal and well understood choice is based on the eigenvalue decomposition of the matrix. Unfortunately, this is computationally very expensive. 

Cheaper methods are based on Gaussian elimination but they require pivoting. We will show how invariant matrix theory provides explicit error formulas for an averaged error based on volume sampling. The formula leads to ratios of elementary symmetric polynomials on the eigenvalues. We discuss some new an old bounds and include several examples where an expected error norm can be computed exactly.
\end{abstract}

\keywords{Low rank matrix approximation \and volume sampling \and kernel methods \and least squares approximation \and regularisation}

\section{\label{sec1} Introduction}

The eigenvalue decomposition of a \emph{real symmetric semi-positive definite matrix $M\in\R^{n,n}$ of rank $r$} is 
\begin{equation}
    \label{eq:evd}
    M = Q \Lambda Q^T
\end{equation} 
where the factor $Q\in\R^{n,r}$ has $r$ orthonormal columns and $\Lambda\in\R^{r,r}$ is a diagonal matrix with elements $\lambda_1 \geq \cdots \geq \lambda_r > 0$. Such matrices are common in applications in machine learning and information retrieval among others~\cite{MahD09}.
The normal matrix $M=X^TX$ which occurs in least squares problems is one instance.

Often $M$ is very large, dense and unstructured. If the rank $r$ is small, however, the eigenvalue decomposition~(\ref{eq:evd}) shows that it is possible to represent $M$ using the $nr+r$ matrix elements of the matrices $Q$ and $\Lambda$. If many of the eigenvalues $\lambda_i\geq 0$ are close to zero one may set these values to zero. So given a matrix $M$ one then might compute the eigenvalue decomposition, remove small eigenvalues and their corresponding eigenvectors from the factors $\Lambda$ and $Q$ to get a good approximation $M_k$ which is represented by $nk+k$ real numbers. In fact, the Eckart-Young-Mirsky theorem~\cite{Dax10} states that the approximation $M_k$ is optimal. Such an approximation leads to improved computational performance when used in algorithms requiring matrix vector products, leads to higher stability and is useful in data analysis (see principle component analysis). However, the computational cost of the eigenvalue decomposition is typically of order $O(n^3)$ requiring the storage of $n^2$ numbers. Thus for large $n$, this approach is often not feasible in practice.

Thus there is a real need to have a faster algorithm which obtains close to optimal approximation of $M$. Here we consider a popular example of such an algorithm which is often referred to CUR or pseudo-skeleton approximation~\cite{GorTZ97}. This algorithm selects $k$ columns (or rows) of the symmetric matrix $M$ and uses them to approximate $M$. More specifically one may represent $M$ by the equation
\begin{equation}\label{eq:block}
    M = P^T\begin{bmatrix} A & B^T \\ B & C \end{bmatrix}P^T
\end{equation}
where $P$ permutes the columns such that the selected ones are moved to the top and $A\in\R^{k,k}$. The CUR approximation then is of the form
\begin{equation}\label{eq:CUR}
    \widehat{M}_k = P^T \begin{bmatrix} A & B^T \\ B & BA^{-1}B^T\end{bmatrix}P.
\end{equation}
One can show that this approximation is of rank $k$ and in the following sections we will investigate the error of this approximation. For the CUR approximation~(\ref{eq:CUR}) to be defined, $A$ needs to be invertible. One approach is to select the $k$ columns for which the determinant of $A$ is maximal. If the matrix $M$ has a rank $r\geq k$ this choice would guarantee that $A$ is invertible. Here we consider an approach which selects the $k$ columns at random with probability proportional to the determinant of $A$. This choice has been termed \emph{volume sampling}~\cite{DesRVW06}. In this case the probability of selecting $k$ columns which lead to a non-invertible $A$ is zero. For this method one can get an exact expression for the expectation of a suitable norm of the error. This is similar to the optimal case. In contrast to the optimal case, however, the determination of the error is computationally often not feasible even if all the eigenvalues are known. Here we will study this further and present upper bounds for these expected errors.

One suitable norm for the error analysis is the nuclear norm (also called trace norm or Schatten 1 norm). The nuclear norm of $M$ is the sum of its singular values. In the case of semi-definite symmetric matrices $M$ the singular values are equal to the eigenvalues and one has 
\begin{equation}\label{eq:nuclear}
    \lVert M \rVert_* = \sum_{i=1}^n \lambda_i
\end{equation}
where $\lVert M \rVert_*$ denotes the nuclear norm. As we will always assume that the eigenvalues are numbered in decreasing order ($\lambda_{i+1}\leq \lambda_i$), the nuclear norm of the error of the optimal approximation is 
\begin{equation}\label{eq:error_opt}
    \lVert M_k - M \rVert_* = \sum_{i=k+1}^n \lambda_i.
\end{equation}

In Section~\ref{sec2} the volume sampling CUR approach is discussed and a formula for the expected error in terms of
matrix invariants is established. The errors as functions of the eigenvalues are further discussed in Section~\ref{sec3} and two special types of matrices are considered in more depth.

\section[Rank k approximation]{\label{sec2} Rank k approximation and expected error in terms of matrix invariants}

In order to establish the framework for volume sampling we introduce the sample space to be the 
symmetric group $\Omega = S_n$. Then any sample $\omega\in S_n$ is a permutation of a set 
with $n$ elements. The symmetric group is the structure of the set of permutation matrices
of arrays with $n$ elements and we will denote the permutation matrix defined by some 
$\omega\in\Omega$ as $P_\omega$. The function mapping $\omega$ to the set of matrices defined
by $M(\omega)= P_\omega M P_\omega^T$ is then a matrix valued random variable. We will denote
the blocks defined in equation~(\ref{eq:block}) of $M(\omega)$ by $A(\omega)$, $B(\omega)$ and
$C(\omega)$ and for simplicity denote the corresponding variables by $M$, $A$, $B$, and $C$.
Finally, we define the probability of some $\omega\in\Omega$ to be
\begin{equation}\label{eq:probability}
    p_k(\omega) := \frac{\det A(\omega)}{\sum_{\xi\in\Omega} \det A (\xi)}.
\end{equation}
With this framework we can now define the expected error of the CUR approximation to be
\begin{equation}\label{eq:error1}
   \E\left(\lVert \widehat{M}_k - M \rVert_*\right) := \sum_{p_k(\omega)_\neq 0} p_k(\omega)\,\lVert C(\omega) - B(\omega)A(\omega)^{-1}B(\omega)^T \lVert_* 
\end{equation}
i.e. the expectation is a sum over permutations $\omega$ with nonzero probability.

The determinant $\det A(\gamma)$ is a $k$-th order principle minor of the unpermuted matrix $M$. On can see
that each principle minor of $M$ occurs $k!(n-k)!$ times when cycling through all the elements of $\gamma\in\Omega$. Let $c_j(M)$ denote the sum all $j$-th principle minors of $M$ in the following theorem.
\begin{theorem}[expected error of volume sampling CUR]\label{thm1}
    \begin{equation}\label{eq:thm1}
        \E\left(\lVert \widehat{M}_k - M \rVert_*\right) = (k+1)\frac{c_{k+1}(M)}{c_k(M)}
    \end{equation}
\end{theorem}

In the proof we will use the lemma:
\begin{lemma}\label{lem2}
If $\begin{bmatrix} A & b \\ b^T & \gamma\end{bmatrix}$ is positive semidefinite and $\det A = 0$ then $$\det \begin{bmatrix} A & b \\ b^T & \gamma\end{bmatrix}=0.$$
\end{lemma}
\begin{proof}
As $\det A = 0$ there exists an $x\neq 0$ such that $Ax=0$. From the semi positive-definiteness one then gets for all $\eta\in\R$
$$\begin{bmatrix} x^T & \eta\end{bmatrix} 
  \begin{bmatrix}A & b \\ b^T & \gamma \end{bmatrix} 
  \begin{bmatrix} x \\ \eta\end{bmatrix} = 2 \eta x^Tb + \eta^2\gamma
  \geq 0$$
  and thus $b^T x = 0$. Thus $\begin{bmatrix}x \\ 0 \end{bmatrix}$ is in the null space of 
  $\begin{bmatrix}A & b \\ b^T & \gamma \end{bmatrix}$.
\end{proof}

\begin{proof}[Proof of  Theorem~\ref{thm1}]
As each minor occurs $k!(n-k)!$ times in the sequence $A(\Omega)$ one has
    $$\sum_{\omega\in\Omega} det A(\omega) = k!(n-k)! c_k(M)$$
and consequently $p_k(\omega) = \frac{\det A(\omega)}{k!(n-k)! c_k(M)}.$
As $C-BA^{-1}B^T$ is positive semi-definite one gets 
    $$\lVert BA^{-1}B^T - C \rVert_* = \trace C - BA^{-1}B^T = 
       \sum_{i=1}^{n-k} c_{i,i} - b_i^T A^{-1} b_i = 
        \frac{\sum_{i=1}^{n-k} \det \begin{bmatrix} A & b_i\\ b_i^T & c_{i,i} \end{bmatrix}}{\det A} $$
     where $b_i^T$ is the $i$-th row of $B$. Multiplying both sides with $\det A$ then summing over all permutations and applying Lemma~\ref{lem2} gives 
    $$ \sum_{p(\omega)\neq 0}(\det A(\omega)) 
        \lVert B(\omega)A(\omega)^{-1}B(\omega)^T-C(\omega)\lVert_* 
        = \sum_{\omega\in\Omega}\sum_{i=1}^{n-k} \
          det \begin{bmatrix}A & b_i \\ b_i^T & c_{i,i}\end{bmatrix}$$
   Interchanging the order of the double sum gives $(k+1)!(n-k)! c_{k+1}(M)$ and inserting
   $p_k(A) = (\det A)/(k!(n-k)! c_k(M))$ from above completes the proof.
\end{proof}

\section[Diagonal matrices]{The case of diagonal matrices or error and stability bounds in terms of the eigenvalues \label{sec3}}

The matrix invariants $c_k(M)$ show up as coefficients of the characteristic polynomial (up to signs). Thus they are also invariant under orthogonal similarity transforms and thus 
        $$c_k(M) = c_k(\Lambda) = e_k(\lambda_1,\ldots,\lambda_n)$$
where $e_k$ are the \emph{elementary symmetric polynomials}. The evaluation of elementary symmetric 
polynomials is computationally very demanding in general. We will thus focus on upper bounds. A simple bound is given in Proposition~\ref{lemX}. This bound is tight and gives a good indication of what the error is for spectra which are rapidly decreasing. For spectra like $\lambda_i=1/i^2$, however, this bound is highly overestimating the error. In this section we will discuss ways to get better bounds for these cases. 
\begin{figure}[t]
\center
\begin{subfigure}[b]{0.45\textwidth}
\begin{adjustbox}{width=\textwidth}
\input{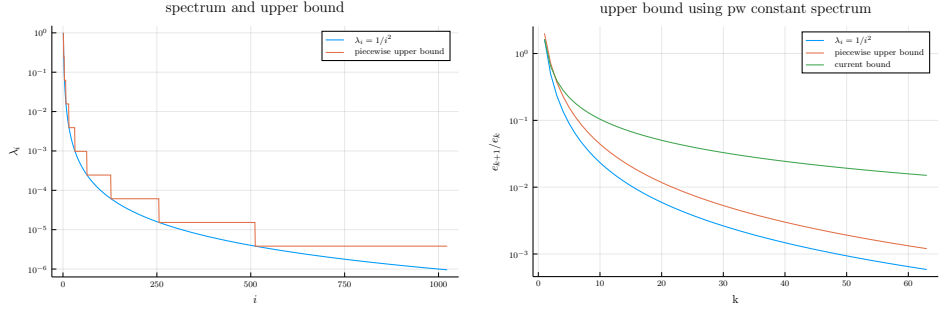}
\end{adjustbox}
\end{subfigure}
\begin{subfigure}[b]{0.45\textwidth}
\begin{adjustbox}{width=\textwidth}
\input{./pwelsymm}\hfill
\end{adjustbox}
\end{subfigure}
\caption{\label{fig:ev} Eigenvalues and the ratios $e_{k+1}/e_k$}
\end{figure}
First we demonstrate the performance of these methods with a computational experiment. Let $\lambda_i = 1/i^2$ for $i=1,\ldots,n$. We then define a piecewise constant array $\mu$ by
$$\mu_i = \frac{1}{4^l}, \quad \text{where $i=2^l+j$, $j=1,\ldots,2^l-1$ and $l=0,\ldots,l_\lmax-1$}.$$
One can show that $\mu_i \geq \lambda_i$ for $i=1,\ldots,n$. From the monotonicity (Lemma~\ref{lem:mono})  one sees that
    $$\frac{e_{k+1}(\lambda)}{e_k(\lambda)} \leq \frac{e_{k+1}(\mu)}{e_k(\mu)}.$$
This is illustrated computationally in Figure~\ref{fig:ev} where the simple bound from Proposition~\ref{lemX} is also displayed to demonstrate the superiority of the new bound. While not as simple as the simple bound, the bound we give here is substantially more accurate and useful for error bounds. In contrast to the values for $\lambda$ the upper bound can be computed even for very large $n$ and medium sized $k$ as explained below.

An upper bound for $e_{k+1}/e_k$ which does not require any properties of the  eigenvalues is obtained from
\begin{proposition}[simple bound]\label{lemX} 
Let $n,k \in \N$, $n>k$ and $\lambda_i\geq 0$ for all $i=1,\ldots,n$. Then  $$e_{k+1}(\lambda_1,\ldots,\lambda_n) \leq e_k(\lambda_1,\ldots,\lambda_n)\, e_1(\lambda_{k+1},\ldots,\lambda_n).$$  
\end{proposition}
  \begin{proof}
  The proof is by induction where the main step is (the second line uses the induction assumption)
  \begin{align*}
      e_{k+1}(\lambda_1,\ldots,\lambda_{n+1}) & = e_k(\lambda_1,\ldots,\lambda_n)\,\lambda_{n+1} + e_{k+1}(\lambda_1,\ldots,\lambda_n) \\
      &\leq e_k(\lambda_1,\ldots,\lambda_n)\, \lambda_{n+1} + e_k(\lambda_1,\ldots,\lambda_n)\,e_1(\lambda_{k+1},\ldots,\lambda_n) \\
      & = e_k(\lambda_1,\ldots,\lambda_n)\, e_1(\lambda_{k+1},\ldots,\lambda_{n+1}).
  \end{align*}
  \end{proof} 
Thus $e_{k+1}/e_k$ is bounded by the optimal error and by Theorem~\ref{thm1} the CUR error is bounded by $k+1$ times the optimal error.

In the case where $\lambda_i = q^i$ for some $q\in (0,1)$ one can compute the elementary symmetric polynomials explicitly as
\begin{proposition}[power eigenvalues]
 \begin{equation}\label{eq:espq}
  e_k(1,q,\ldots,q^{n-1}) = q^{k(k-1)/2}\, \prod_{i=1}^k \frac{1-q^{n-i+1}}{1-q^i}, \quad \text{for all $k,n \in \N$}.   
 \end{equation}
\end{proposition}

\begin{proof}
  Let $e_{k,n}:= e_k(1,q,\ldots,q^{n-1})$. We then use induction over $n$ to show that equation~(\ref{eq:espq}) holds for all $k$.
  
  First, for $n=1$ one has $e_{0,1}=e_{1,1}=1$ and $e_{k,1}=0$ by definition of $e_k$. Thus the claimed result holds for $n=1$.
  
  We now have to show that if equation~(\ref{eq:espq}) holds for some $n$ and all $k$ it also holds when $n$ is replaced by $n+1$. For this we use the following recursion for the symmetric elementary polynomials:
  $$e_{k,n+1} = q^n e_{k-1,n} + e_{k,n}.$$
  Then one verifies that the claimed equation~(\ref{eq:espq}) is equivalent to
  $$e_{k,n} = \prod_{i=1}^k \frac{q^{i-1}-q^n}{1-q^i}.$$
  It follows that
  $$e_{k,n} = \frac{q^{k-1}-q^n}{1-q^k}\,  e_{k-1,n}.$$
  One then gets
  \begin{align*}
    e_{k,n+1} &= q^n e_{k-1,n} + e_{k,n} = \left(q^n+\frac{q^{k-1}-q^n}{1-q^k}\right) \,  e_{k-1,n} \\ &=
    \frac{q^{k-1}(1-q^{n+1})}{1-q^k}
    \,  e_{k-1,n} =
    \frac{1-q^{n+1}}{1-q^k}\,
  \prod_{i=1}^{k-1}\frac{q^i-q^{n+1}}{1-q^i}\\
    &= \prod_{i=1}^k \frac{q^{i-1}-q^{n+1}}{1-q^i}.
  \end{align*}
\end{proof}
The ratio $e_{k+1}/e_k=\frac{q^k-q^n}{1-q^{k+1}}$ obtained from this result then leads to the expected error $$(k+1)\frac{e_{k+1}}{e_k} = \frac{k+1}{1+q+\cdots+ q^k}(q^k+\cdots + q^n).$$ 
One sees that the CUR method is competitive for small $q$ but not for $q\approx 1$.

We now derive some results which are useful in the derivation and computation of error bounds. First we pad the eigenvalue vectors with zeros so that $\lambda\in \ell_0$, the set of series which are nonzero for finitely many indices.

 For any $k$ we define the head of $\lambda$ to be $\lambda^h = (\lambda_1,\ldots,\lambda_k,0,\ldots)$ and the tail $\lambda^t=(\lambda_{k+1},\ldots,\lambda_n,0,\ldots)$.
We introduce a concatenation of two spectra $\lambda$ and $\mu$ denoted by $[\lambda,\mu]$ where
$$[\lambda,\mu] = (\lambda_1,\ldots,\lambda_n,\mu_1,\ldots,\mu_m,0,\ldots)$$
which potentially is reordered for size but in the cases considered here we have $\mu_1 \leq \lambda_n$.
Thus $\lambda=[\lambda^h,\lambda^t]$. The decrease of the tail is modelled by the sequence $\rho$ with
$$\rho_i = \lambda^t/\lambda_{k+1}.$$
We then introduce a function $f: \R_0^\infty \rightarrow \R^{k+2}$ with components
$$f_i(\lambda) = e_i(\lambda).$$
We now introduce the \emph{convolution} of two elements of $\R_+^m$ by 
$$(u*v)_i = \sum_{i+j = k} u_i v_j. $$
As the $e_i$ are coefficients of a characteristic polynomial one has
\begin{lemma}[convolution theorem]
$$f([\lambda,\mu]) = f(\lambda)*f(\mu).$$
\end{lemma}
The next lemma is a consequence of the fact that $f_i$ is a homogeneous $i-1$st degree polynomial.
\begin{lemma}[scaling lemma]
  The $i$-th component $f_i$ of $f$ satisfies
  $$f_i(s\lambda) = s^{i-1}f_i(\lambda).$$
\end{lemma}
One then has
\begin{proposition}
  Let $u^h=f(\lambda^h)$ and $w=f(\rho)$. Then
  $$\frac{e_{k+1}(\lambda)}{e_k(\lambda)} =
  \gamma \lambda_{k+1},$$
  where
  $$\gamma = \frac{\sum_{i=0}^k \lambda_{k+1}^{k-i}\; w_{k+1-i} u^h_i}{\sum_{i=0}^k\; \lambda_{k+1}^{k-i} w_{k-i} u^h_i}.$$
\end{proposition}
\begin{proof}
 By the convolution theorem and the definition of $\rho$ one has
 $$u = u^t * u^h$$
 where $u=f(\lambda)$ and $u^t = f(\lambda^t)=f(\lambda_{k+1}\rho)$.

    Using the scaling lemma and $u^h_{k+1}=0$ one then gets
    $$\frac{u_{k+1}}{u_k} 
    = \lambda_{k+1}\frac{\sum_{i=0}^k \lambda_{k+1}^{k-i}\; w_{k+1-i} u^h_i}{\sum_{i=0}^k\; \lambda_{k+1}^{k-i} w_{k-i} u^h_i}.$$
\end{proof}

We will now show results used to obtain bounds for the case of slowly decreasing $\rho_i$. In these cases one observes that the sequence $w_{i+1}/w_i$ first increases before it decreases. The main tool to obtain bounds is the monotonicity in of the ratios $e_{k+1}(\lambda)/e_k(\lambda)$ in $\lambda$.

\begin{lemma}[monotonicity]\label{lem:mono}
Let $0 < \lambda \leq \mu$ component wise then
$$\frac{e_{k+1}(\lambda)}{e_k(\lambda)} \leq \frac{e_{k+1}(\mu)}{e_k(\mu)}.$$
\end{lemma}
This lemma can be proven directly but is also a consequence of a result by Marcus and Lopes~\cite{marl57}:
$$\frac{e_{k+1}(\lambda+\mu)}{e_k(\lambda+\mu)} \geq \frac{e_{k+1}(\lambda)}{e_k(\lambda)} + \frac{e_{k+1}(\mu)}{e_k(\mu)}.$$ 
This holds for $\mu$ and $\lambda$ being nonegative.

A direct application of the convolution theorem gives the representation
\begin{proposition}
  \begin{equation}\label{conv1}
    f(\lambda) = {u^{l_\lmax}}*\cdots*u^0  
  \end{equation}
  where $u^l=f(q^l\epsilon_{2^l})$ has the components
  $$u^l_j = q^{l(j-1)}\binom{2^k}{j-1},\quad j=1,\ldots,2^l.$$
\end{proposition}
This proposition is used to show that the ratio $e_{k+1}/e_k$ is computationally feasible for the piecewise spectrum $\mu$ used in Figure~\ref{fig:ev}. 
The complexity of computing the $l_\lmax$ convolutions of size $k+1$ is of order $O(l_\lmax (k+1)^2)$. Using this formula is typically much faster than using the standard recursions which results in a total complexity of $O(n(k+1))$ for very large $n$. One observes that often only a small number of components of the $u^j$ are substantially different from zero so that the complexity can be further reduced.

\section{Conclusion}
While the CUR method combined with volume sampling admits an explicit and exact formula for the expected approximation error of the computed rank $k$ approximation there is little known about the performance of the method for slowly decreasing eigenvalues. This paper provides new error bounds for the CUR method with volume sampling which shows that this approach is competitive with the optimal approach for slowly decreasing eigenvalues even for large approximation ranks.

reviews the basic theory and provides a framework and some simple examples which provide some insights on how the CUR method performs well in particular for these cases of slowly decreasing eigenvalues. Future work may consider the effect of the initial eigenvalues and the rate of decrease of the tail of the spectrum in more detail using the convolution formula provided here.

\printbibliography

\end{document}

%% file: pwelsymm.tex
\begin{tikzpicture}[/tikz/background rectangle/.style={fill={rgb,1:red,1.0;green,1.0;blue,1.0}, draw opacity={1.0}}, show background rectangle]
\begin{axis}[point meta max={nan}, point meta min={nan}, legend cell align={left}, title={upper bound using pw constant spectrum}, title style={at={{(0.5,1)}}, anchor={south}, font={{\fontsize{14 pt}{18.2 pt}\selectfont}}, color={rgb,1:red,0.0;green,0.0;blue,0.0}, draw opacity={1.0}, rotate={0.0}}, legend style={color={rgb,1:red,0.0;green,0.0;blue,0.0}, draw opacity={1.0}, line width={1}, solid, fill={rgb,1:red,1.0;green,1.0;blue,1.0}, fill opacity={1.0}, text opacity={1.0}, font={{\fontsize{8 pt}{10.4 pt}\selectfont}}, at={(0.98, 0.98)}, anchor={north east}}, axis background/.style={fill={rgb,1:red,1.0;green,1.0;blue,1.0}, opacity={1.0}}, anchor={north west}, xshift={1.0mm}, yshift={-1.0mm}, width={150.4mm}, height={99.6mm}, scaled x ticks={false}, xlabel={k}, x tick style={color={rgb,1:red,0.0;green,0.0;blue,0.0}, opacity={1.0}}, x tick label style={color={rgb,1:red,0.0;green,0.0;blue,0.0}, opacity={1.0}, rotate={0}}, xlabel style={at={(ticklabel cs:0.5)}, anchor=near ticklabel, font={{\fontsize{11 pt}{14.3 pt}\selectfont}}, color={rgb,1:red,0.0;green,0.0;blue,0.0}, draw opacity={1.0}, rotate={0.0}}, xmajorgrids={true}, xmin={-0.8599999999999999}, xmax={64.86}, xtick={{0.0,10.0,20.0,30.0,40.0,50.0,60.0}}, xticklabels={{$0$,$10$,$20$,$30$,$40$,$50$,$60$}}, xtick align={inside}, xticklabel style={font={{\fontsize{8 pt}{10.4 pt}\selectfont}}, color={rgb,1:red,0.0;green,0.0;blue,0.0}, draw opacity={1.0}, rotate={0.0}}, x grid style={color={rgb,1:red,0.0;green,0.0;blue,0.0}, draw opacity={0.1}, line width={0.5}, solid}, axis x line*={left}, x axis line style={color={rgb,1:red,0.0;green,0.0;blue,0.0}, draw opacity={1.0}, line width={1}, solid}, scaled y ticks={false}, ylabel={$e_{k+1}/e_k$}, y tick style={color={rgb,1:red,0.0;green,0.0;blue,0.0}, opacity={1.0}}, y tick label style={color={rgb,1:red,0.0;green,0.0;blue,0.0}, opacity={1.0}, rotate={0}}, ylabel style={at={(ticklabel cs:0.5)}, anchor=near ticklabel, font={{\fontsize{11 pt}{14.3 pt}\selectfont}}, color={rgb,1:red,0.0;green,0.0;blue,0.0}, draw opacity={1.0}, rotate={0.0}}, ymode={log}, log basis y={10}, ymajorgrids={true}, ymin={0.00045890342403739746}, ymax={2.550307782967796}, ytick={{0.001,0.01,0.1,1.0}}, yticklabels={{$10^{-3}$,$10^{-2}$,$10^{-1}$,$10^{0}$}}, ytick align={inside}, yticklabel style={font={{\fontsize{8 pt}{10.4 pt}\selectfont}}, color={rgb,1:red,0.0;green,0.0;blue,0.0}, draw opacity={1.0}, rotate={0.0}}, y grid style={color={rgb,1:red,0.0;green,0.0;blue,0.0}, draw opacity={0.1}, line width={0.5}, solid}, axis y line*={left}, y axis line style={color={rgb,1:red,0.0;green,0.0;blue,0.0}, draw opacity={1.0}, line width={1}, solid}]
    \addplot[color={rgb,1:red,0.0;green,0.6056;blue,0.9787}, name path={cb67674d-4c55-4da8-bee2-71241db967dd}, draw opacity={1.0}, line width={1}, solid]
        table[row sep={\\}]
        {
            \\
            1.0  1.643957027355849  \\
            2.0  0.49279617636918704  \\
            3.0  0.23447861416973947  \\
            4.0  0.1366705134175898  \\
            5.0  0.0893859003535539  \\
            6.0  0.06297832059164772  \\
            7.0  0.04674658992847242  \\
            8.0  0.03606230388099685  \\
            9.0  0.02865820562265044  \\
            10.0  0.02331729258480508  \\
            11.0  0.019338767474978424  \\
            12.0  0.01629594448633599  \\
            13.0  0.0139169714140992  \\
            14.0  0.012021992689308949  \\
            15.0  0.010488171057587564  \\
            16.0  0.009229296392898979  \\
            17.0  0.008183421667600003  \\
            18.0  0.007305107704860726  \\
            19.0  0.006560412972255453  \\
            20.0  0.0059235716275351435  \\
            21.0  0.005374739642518619  \\
            22.0  0.004898433838012207  \\
            23.0  0.0044824306579428555  \\
            24.0  0.004116976208631558  \\
            25.0  0.003794210932799392  \\
            26.0  0.0035077447710343793  \\
            27.0  0.0032523394519814734  \\
            28.0  0.003023668116405873  \\
            29.0  0.002818131488370666  \\
            30.0  0.0026327158874265566  \\
            31.0  0.00246488254239199  \\
            32.0  0.0023124805624267024  \\
            33.0  0.002173677958842506  \\
            34.0  0.0020469065626996215  \\
            35.0  0.001930817728960583  \\
            36.0  0.0018242464792479676  \\
            37.0  0.001726182294927793  \\
            38.0  0.0016357451875226602  \\
            39.0  0.001552165984292364  \\
            40.0  0.0014747700013907424  \\
            41.0  0.0014029634554028091  \\
            42.0  0.0013362221007379763  \\
            43.0  0.001274081685790663  \\
            44.0  0.0012161299026559795  \\
            45.0  0.0011619995691693392  \\
            46.0  0.0011113628323341695  \\
            47.0  0.0010639262219645797  \\
            48.0  0.00101942641497567  \\
            49.0  0.0009776265960068985  \\
            50.0  0.0009383133203404303  \\
            51.0  0.000901293801434365  \\
            52.0  0.0008663935586474119  \\
            53.0  0.0008334543715214423  \\
            54.0  0.0008023324958074728  \\
            55.0  0.0007728971036571829  \\
            56.0  0.0007450289163636187  \\
            57.0  0.0007186190029635791  \\
            58.0  0.0006935677221038203  \\
            59.0  0.0006697837879785219  \\
            60.0  0.0006471834439897424  \\
            61.0  0.000625689730166411  \\
            62.0  0.000605231832381057  \\
            63.0  0.0005857445030928741  \\
        }
        ;
    \addlegendentry {$\lambda_i = 1/i^2$}
    \addplot[color={rgb,1:red,0.8889;green,0.4356;blue,0.2781}, name path={8e356226-8f96-46fe-9bb1-80676bf2497c}, draw opacity={1.0}, line width={1}, solid]
        table[row sep={\\}]
        {
            \\
            1.0  1.998046875  \\
            2.0  0.7130298614501953  \\
            3.0  0.37326323671392986  \\
            4.0  0.22986569570717746  \\
            5.0  0.15591509906967999  \\
            6.0  0.11275673883128429  \\
            7.0  0.08534419310669748  \\
            8.0  0.06683652911363497  \\
            9.0  0.05375218480235358  \\
            10.0  0.04416173280587073  \\
            11.0  0.03692369960295825  \\
            12.0  0.031327232511392696  \\
            13.0  0.02691091135797455  \\
            14.0  0.02336472148042688  \\
            15.0  0.02047417513820501  \\
            16.0  0.01808706393193113  \\
            17.0  0.016092951098730806  \\
            18.0  0.01441012399086337  \\
            19.0  0.012977064248634937  \\
            20.0  0.011746734865822693  \\
            21.0  0.010682669721216008  \\
            22.0  0.009756243639768411  \\
            23.0  0.008944732184164068  \\
            24.0  0.008229910077331527  \\
            25.0  0.007597023585441209  \\
            26.0  0.007034026813289545  \\
            27.0  0.00653100706984395  \\
            28.0  0.006079747571622262  \\
            29.0  0.005673391182738928  \\
            30.0  0.005306179360335853  \\
            31.0  0.004973247684904276  \\
            32.0  0.004670464391134736  \\
            33.0  0.004394301878454922  \\
            34.0  0.004141733732793682  \\
            35.0  0.003910151640076004  \\
            36.0  0.003697297925556934  \\
            37.0  0.0035012104538248087  \\
            38.0  0.0033201773709370525  \\
            39.0  0.003152699731977449  \\
            40.0  0.0029974604834692403  \\
            41.0  0.0028532985957308164  \\
            42.0  0.0027191873908685464  \\
            43.0  0.002594216306233244  \\
            44.0  0.0024775754844920343  \\
            45.0  0.0023685427001205726  \\
            46.0  0.002266472225680543  \\
            47.0  0.0021707853154146087  \\
            48.0  0.002080962042786812  \\
            49.0  0.0019965342759130367  \\
            50.0  0.001917079612890511  \\
            51.0  0.0018422161297938773  \\
            52.0  0.001771597819069624  \\
            53.0  0.0017049106164065226  \\
            54.0  0.0016418689308090812  \\
            55.0  0.001582212606279033  \\
            56.0  0.0015257042547891714  \\
            57.0  0.0014721269095695252  \\
            58.0  0.001421281955479994  \\
            59.0  0.0013729872997065343  \\
            60.0  0.0013270757514223786  \\
            61.0  0.0012833935835899403  \\
            62.0  0.0012417992538945604  \\
            63.0  0.0012021622650220706  \\
        }
        ;
    \addlegendentry {piecewise upper bound}
    \addplot[color={rgb,1:red,0.2422;green,0.6433;blue,0.3044}, name path={88584b66-2bf7-499b-9bf9-360a3184c83a}, draw opacity={1.0}, line width={1}, solid]
        table[row sep={\\}]
        {
            \\
            1.0  1.6439570273558495  \\
            2.0  0.643957027355848  \\
            3.0  0.39395702735584776  \\
            4.0  0.28284591624473654  \\
            5.0  0.2203459162447367  \\
            6.0  0.18034591624473673  \\
            7.0  0.152568138466959  \\
            8.0  0.13215997520165285  \\
            9.0  0.11653497520165282  \\
            10.0  0.10418929618930714  \\
            11.0  0.09418929618930713  \\
            12.0  0.08592483337938979  \\
            13.0  0.07898038893494534  \\
            14.0  0.07306322917163174  \\
            15.0  0.06796118835530518  \\
            16.0  0.06351674391086073  \\
            17.0  0.05961049391086077  \\
            18.0  0.056150286298404024  \\
            19.0  0.0530638665453176  \\
            20.0  0.05029378344282452  \\
            21.0  0.047793783442824526  \\
            22.0  0.045526209746679404  \\
            23.0  0.04346009404420007  \\
            24.0  0.041569734875958095  \\
            25.0  0.03983362376484699  \\
            26.0  0.03823362376484699  \\
            27.0  0.036754333824018576  \\
            28.0  0.035382591711535724  \\
            29.0  0.0341070815074541  \\
            30.0  0.032918020865361344  \\
            31.0  0.03180690975425023  \\
            32.0  0.030766327027923483  \\
            33.0  0.02978976452792349  \\
            34.0  0.028871490882377115  \\
            35.0  0.028006438979262926  \\
            36.0  0.02719011244865068  \\
            37.0  0.026418507510379077  \\
            38.0  0.025688047320459425  \\
            39.0  0.02499552654483616  \\
            40.0  0.024338064348912426  \\
            41.0  0.02371306434891243  \\
            42.0  0.023118180351291963  \\
            43.0  0.022551286927255682  \\
            44.0  0.022010454044616413  \\
            45.0  0.02149392511899658  \\
            46.0  0.02100009795850275  \\
            47.0  0.020527508166442266  \\
            48.0  0.02007481463995969  \\
            49.0  0.01964078686218191  \\
            50.0  0.019224293734318523  \\
            51.0  0.018824293734318525  \\
            52.0  0.01843982622182333  \\
            53.0  0.01807000373661623  \\
            54.0  0.01771400516061053  \\
            55.0  0.01737106963248982  \\
            56.0  0.017040491120093128  \\
            57.0  0.016721613569072723  \\
            58.0  0.0164138265576846  \\
            59.0  0.016116561397161413  \\
            60.0  0.015829287625256785  \\
            61.0  0.015551509847479007  \\
            62.0  0.015282764886447028  \\
            63.0  0.015022619204865343  \\
        }
        ;
    \addlegendentry {current bound}
\end{axis}
\end{tikzpicture}

%% file: manu.bib
@article{Dax10,
	title = {On extremum properties of orthogonal quotients matrices},
	volume = {432},
	issn = {0024-3795},
	url = {http://www.sciencedirect.com/science/article/pii/S0024379509005485},
	doi = {10.1016/j.laa.2009.10.034},
	pages = {1234--1257},
	number = {5},
	journaltitle = {Linear Algebra and its Applications},
	shortjournal = {Linear Algebra and its Applications},
	author = {Dax, Achiya},
	urldate = {2020-10-14},
	date = {2010-02-15},
	langid = {english}
}

@article{MahD09,
	title = {{CUR} matrix decompositions for improved data analysis},
	volume = {106},
	issn = {0027-8424, 1091-6490},
	url = {http://www.pnas.org/lookup/doi/10.1073/pnas.0803205106},
	doi = {10.1073/pnas.0803205106},
	pages = {697--702},
	number = {3},
	journaltitle = {Proceedings of the National Academy of Sciences},
	shortjournal = {{PNAS}},
	author = {Mahoney, Michael W. and Drineas, Petros},
	urldate = {2020-04-13},
	date = {2009-01-20},
	langid = {english}}

@inproceedings{DesRVW06,
author = {Deshpande, Amit and Rademacher, Luis and Vempala, Santosh and Wang, Grant},
title = {Matrix Approximation and Projective Clustering via Volume Sampling},
year = {2006},
isbn = {0898716055},
publisher = {Society for Industrial and Applied Mathematics},
address = {USA},
booktitle = {Proceedings of the Seventeenth Annual ACM-SIAM Symposium on Discrete Algorithm},
pages = {1117–1126},
numpages = {10},
location = {Miami, Florida},
series = {SODA '06}
}

@article{GorTZ97,
title = "A theory of pseudoskeleton approximations",
journal = "Linear Algebra and its Applications",
volume = "261",
number = "1",
pages = "1 - 21",
year = "1997",
issn = "0024-3795",
doi = "https://doi.org/10.1016/S0024-3795(96)00301-1",
url = "http://www.sciencedirect.com/science/article/pii/S0024379596003011",
author = "S.A. Goreinov and E.E. Tyrtyshnikov and N.L. Zamarashkin"
}

@article{marl57, title={Inequalities for Symmetric Functions and Hermitian Matrices}, volume={9}, DOI={10.4153/CJM-1957-037-9}, journal={Canadian Journal of Mathematics}, publisher={Cambridge University Press}, author={Marcus, M. and Lopes, L.}, year={1957}, pages={305–312}}
